\documentclass[12pt,a4paper]{amsart}
\usepackage{enumerate}

\usepackage{amsmath,amssymb,xspace,amsthm}
\newtheorem{theorem}{Theorem}
\newtheorem{lemma}[theorem]{Lemma}
\newtheorem{corollary}[theorem]{Corollary}
\newtheorem{conjecture}[theorem]{Conjecture}
\numberwithin{equation}{section}

\renewcommand{\a}{\ensuremath{\alpha}}

\newcommand{\Z}{\ensuremath{\mathbb{Z}}\xspace}

\newcommand{\supp}{\ensuremath{\operatorname{Supp}}\xspace}

\renewcommand{\phi}{\varphi}

\renewcommand{\leq}{\leqslant}
\renewcommand{\geq}{\geqslant}

\begin{document}
\title{Supports of weight modules over Witt algebras}
\author{Volodymyr Mazorchuk and Kaiming Zhao}
\date{}
\maketitle

\begin{abstract}
In this paper,  as the first step towards classification of simple
weight modules with finite dimensional weight spaces over Witt
algebras $W_n$, we explicitly describe supports of such modules. We
also obtain some descriptions on the support of an arbitrary simple
weight module over a $\Z^n$-graded Lie algebra $\mathfrak{g}$ having a 
root space decomposition $\oplus_{\alpha\in\Z^n}\mathfrak {g}_\alpha$ 
with respect to the abelian subalgebra $\mathfrak {g}_0$, 
with the property $[\mathfrak{g}_\alpha,\mathfrak {g}_\beta]=
\mathfrak {g}_{\alpha+\beta}$ for
all $\alpha,\beta\in\Z^n$, $\alpha\neq \beta$ (this class 
contains the algebra $W_n$).
\end{abstract}


\vskip 5pt
\noindent
{\em 2000  Math. Subj. Class.:}
17B10, 17B20, 17B65, 17B66, 17B68

\vskip 10pt

\section{Introduction}\label{s0}

Classification of simple weight modules is a classical problem
in the representation theory of Lie algebras. Simple weight
modules with finite-dimensional weight spaces (sometimes also
called Harish-Chandra modules) are classified for several classes
of algebras, including simple finite-dimensional algebras
(\cite{M2}) and various generalized Virasoro algebras
(\cite{M,Ma25,LZ}). In the general case, however, the problem
is solved only for the Lie algebra $\mathfrak{sl}_2$ (\cite{Ga}, see
\cite[Chapter~3]{Ma5} for details).

A first step in understanding simple weight modules is to understand
possible forms for the support of a module. For simple finite-dimensional
Lie algebras this was originated in  \cite{Fe,Fu} and completed in
\cite{DMP}, where it was shown that any simple weight module is either
dense (that is has the maximal possible support) or is the quotient
of a parabolically induced module. For some algebras of Cartan type
a similar result was obtained in \cite{PS}.

A new effect appears for generalizations of the Virasoro algebra. In
this case already in the class of Harish-Chandra modules there are
modules whose support is one element less than the maximal possible
(in what follows we call such modules punctured). Description of the
support for weight modules over the Virasoro algebra is relatively
easy (see \cite{Ma3}) and for all other algebras there are only some
partial results (\cite{Ma1,Ma2}).

The original motivation for the present paper is the problem of
classification of simple Harish-Chandra modules for the Witt algebra
$W_n$ (the algebra of derivations of a Laurent polynomial algebra in
$n$ commuting variables). For the moment this problem seems to be
too difficult, so as a first step in the present paper we completely
determine the support of such modules, basically reducing the
original classification problem to two cases: classification of
simple dense modules and classification of simple punctured modules.
The main result of the present paper asserts that any simple
Harish-Chandra $W_n$-module is either dense (with uniformly bounded
weight spaces) or punctured (with uniformly bounded weight spaces)
or is the simple quotient of some generalized Verma module. The
latter class of modules is relatively well-understood (\cite{BZ}).
It is known that both dense and punctured modules do exist
(\cite{Sh,Ra}). The main result of the paper is formulated and
proved in Section~\ref{s2} after some preliminaries collected in
Section~\ref{s1}.

We also obtain some partial results about the form of the support
for arbitrary weight modules over $\Z^n$-graded Lie algebras
$\mathfrak {g}=\oplus_{\alpha\in\Z^n}\mathfrak {g}_\alpha$ (root
space decomposition with respect to the abelian subalgebra
$\mathfrak {g}_0$) with the property $[\mathfrak
{g}_\alpha,\mathfrak {g}_\beta]=\mathfrak {g}_{\alpha+\beta}$ for
all $\alpha,\beta\in\Z^n$ with $\alpha\neq \beta$, which generalize
the results in \cite{Ma1,Ma2}. In fact, we recover and give a more
detailed proof for the latter results. In particular, we show that
under some mild additional assumptions the complement (in the weight
lattice) to the support of a simple module is either very small
(roughly speaking belongs to a sublattice of dimension $n-2$) or is
at least a half  of the lattice, see Theorem 8. This is done in
Subsection~\ref{s3.1}. The case $n=2$ (related to \cite{Ma1,Ma2}) is
studied in details in Subsection~\ref{s3.2}.

We finish the paper with a brief description of similar results for
so-called mixed modules. The support of a module can be refined to
encode the information about finite-dimensional and
infinite-dimensional weight spaces separately. A mixed module is a
module which contains both finite- and infinite-dimensional weight
spaces in the same coset of a weight lattice. In \cite{MZ} it was
shown that there are no simple mixed modules over the Virasoro
algebra. For $W_n$, $n>1$, mixed modules do exist. However, in
Subsection~\ref{s3.3} we show that the part of the support of the
mixed module, which describes infinite-dimensional weight spaces,
behaves similarly to the support of a Harish-Chandra module. In
particular, we deduce that under some mild assumptions the support
of a mixed module is contained in a half of the weight lattice. We
conjecture that any mixed $W_n$-module is neither dense nor
punctured.

\section{Notation and preliminaries}\label{s1}

\subsection{Weight modules over Witt algebras}\label{s1.1}

We denote by $\mathbb{Z}$, $\mathbb{Z}_+$, $\mathbb{N}$ and $\mathbb{Q}$
the sets of  all integers, nonnegative integers, positive integers
and rational numbers, respectively.
For any set $X$ every $\alpha\in X^n$ has the form $\alpha=
(\alpha_1,\alpha_2,\dots,\alpha_n)$, where $\alpha_i\in X$ for all $i$.
For a Lie algebra $\mathfrak{a}$ we denote by $U(\mathfrak{a})$
the corresponding universal enveloping algebra.

Let $\Bbbk$ denote an algebraically closed field of characteristic zero.
For a positive integer $n$ the corresponding
classical Witt algebra $\mathfrak{W}_n$ is defined as the algebra
of derivations of the Laurent polynomial algebra
$\Bbbk[t_1^{\pm1},\dots,t_n^{\pm1}]$ in $n$ commuting variables
$t_1, t_2,\dots,t_n$. The algebra $\mathfrak{W}_n$ is simple and
for $n=1$ the algebra $\mathfrak{W}_1$ is the centerless Virasoro algebra.

We fix a positive integer $n$ and denote $\mathfrak{g}=\mathfrak{W}_n$.
For $i\in\{1,2,\dots,n\}$ set $\partial _i=t_i\frac{\partial}{\partial t_i}$
and denote by $\mathfrak {g}_0$ the $\Bbbk$-linear span of all the
$\partial _i$'s. Then $\mathfrak {g}_0$ is an abelian Lie subalgebra
of $\mathfrak{g}$, called the {\em Cartan subalgebra}.

For any $\alpha\in\mathbb{Z}^n$
set $t^\alpha=t_1^{\alpha_1}t_2^{\alpha_2}\cdots t_n^{\alpha_n}$.
If $\partial\in\mathfrak{g}_0$ is arbitrary, then
$t^\alpha\partial\in\mathfrak{g}$. Setting  $\mathfrak{g}_\alpha=
t^\alpha\mathfrak{g}_0$ we obtain the following  decomposition
of $\mathfrak{g}$:
\begin{equation}\label{eqno1}
\mathfrak{g}=\bigoplus_{\alpha\in\mathbb{Z}^n}\mathfrak{g}_\alpha.
\end{equation}
It is easy to check that $[\mathfrak{g}_\alpha,\mathfrak{g}_\beta]\subset
\mathfrak{g}_{\alpha+\beta}$
(and even $[\mathfrak{g}_\alpha,\mathfrak{g}_\beta]=
\mathfrak{g}_{\alpha+\beta}$ unless $\alpha=\beta$)
and hence the above decomposition,
in fact, induces a $\mathbb{Z}^n$-grading of $\mathfrak{g}$.
The adjoint action of $\mathfrak{g}_0$ on $\mathfrak{g}$ is
diagonalizable and the decomposition \eqref{eqno1} coincides
with the decomposition of $\mathfrak{g}$ into a direct sum
of $\mathfrak{g}_0$-eigenspaces.

If $\gamma\in\Bbbk^n$ is such that
$\gamma_1, \gamma_2,\dots, \gamma_n$ are linearly independent over $\mathbb{Q}$, then the subalgebra
\begin{displaymath}
\mathrm{Vir}(\gamma)=\Bbbk\langle t^\beta(\gamma_1\partial_1+
\gamma_2\partial_2+\cdots+ \gamma_n\partial_n)\,:\,\beta\in\Z^n
\rangle
\end{displaymath}
of $\mathfrak{g}$ is a centerless {\em generalized}
(or {\em higher rank}) Virasoro algebra of rank $n$ (in the
sense of \cite{PZ}). For any subgroup $G$ of $\mathbb{Z}^n$
there is also the corresponding subalgebra
$\displaystyle\mathfrak{g}(G)=\bigoplus_{\alpha\in G}
\mathfrak{g}_{\alpha}$  of $\mathfrak{g}$.

A $\mathfrak{g}$-module $V$ is called a {\em weight} module
provided that the action of $\mathfrak{g}_0$ on $V$ is
diagonalizable. For example, from the previous paragraph it follows
that the adjoint $\mathfrak{g}$-module is a weight module. For any
weight module $V$ we have the decomposition
\begin{equation}\label{eqno2}
V=\bigoplus_{\lambda\in \mathfrak{g}_0^*}V_{\lambda},
\end{equation}
where $\mathfrak{g}_0^*=\mathrm{Hom}_{\Bbbk}(\mathfrak{g}_0,\Bbbk)$
and
\begin{displaymath}
V_{\lambda}=\{v\in V:\partial v=\lambda(\partial)v
\text{ for all }\partial\in \mathfrak{g}_0\}.
\end{displaymath}
The space $V_{\lambda}$ is called the {\em weight space} corresponding
to the {\em weight} $\lambda$. The {\em support} $\mathrm{supp}(V)$
of the weight module
$V$ is defined as the set of all weights $\lambda$ for which
$V_{\lambda}\neq 0$. If $V$ is a weight $\mathfrak{g}$-module
and $\dim_{\Bbbk}V_{\lambda}<\infty$ for all
$\lambda\in \mathfrak{g}_0^*$, the module
$V$ is called a {\em Harish-Chandra} module.

We consider $\mathbb{Z}^n$ as a subset of $\mathfrak{g}_0^*$ such
that each $\alpha\in \mathbb{Z}^n$ becomes the weight of the weight
space $\mathfrak{g}_{\alpha}$ in the adjoint $\mathfrak{g}$-module
(i.e. $\a(\partial_i)=\a_i$ for all $\alpha \in \mathbb{Z}^n$).
Under this convention, the decomposition \eqref{eqno1} becomes a
special case of the decomposition \eqref{eqno2} for the adjoint
$\mathfrak{g}$-module. Furthermore, if $V$ is a weight
$\mathfrak{g}$-module, then for all $\alpha\in \mathbb{Z}^n$ and
$\lambda\in \mathfrak{g}_0^*$ we have
$\mathfrak{g}_{\alpha}V_{\lambda}\subset V_{\lambda+\alpha}$. From
this it follows that if $V$ is an indecomposable weight module  (in
particular, simple), then $\mathrm{supp}(V)\subset
\lambda+\mathbb{Z}^n$ for some $\lambda\in \mathfrak{g}_0^*$.

A weight $\mathfrak{g}$-module $V$ is called
\begin{itemize}
\item {\em dense} provided that $\mathrm{supp}(M)=\lambda+\mathbb{Z}^n$
for some $\lambda\in\mathfrak {g}_0^*$;
\item {\em punctured} provided that $\mathrm{supp}(M)=
\mathbb{Z}^n\setminus\{0\}$;
\item {\em uniformly bounded} provided that there exists a positive
integer $N$ such that $\dim V_\lambda<N$ for all
$\lambda\in\mathfrak {g}_0^*$.
\end{itemize}

\subsection{Highest weight modules over Witt algebras}\label{s1.2}

Choose some subgroup $G\subset \mathbb{Z}^n$ and some nonzero
$\beta\in \mathbb{Z}^n$ such that $\mathbb{Z}^n\cong G\oplus H$,
where $H$ is the subgroup of $\mathbb{Z}^n$, generated by
$\beta$. Define the following subalgebras of $\mathfrak{g}$:
\begin{displaymath}
\mathfrak{a}_G:=\bigoplus_{\alpha\in G}\mathfrak{g}_{\alpha};\quad
\mathfrak{n}^+_G:=\bigoplus_{\alpha\in G,k\in\mathbb{N}}
\mathfrak{g}_{\alpha+k\beta};\quad
\mathfrak{n}^-_G:=\bigoplus_{\alpha\in G,k\in\mathbb{N}}
\mathfrak{g}_{\alpha-k\beta}.
\end{displaymath}
This gives the triangular decomposition
$\mathfrak{g}=\mathfrak{n}^-_G\oplus \mathfrak{a}_G
\oplus \mathfrak{n}^+_G$ and allows us to define highest weight
$\mathfrak{g}$-modules with respect to this decomposition.
Note that the algebras $\mathfrak{n}^-_G$ and
$\mathfrak{n}^+_G$ do not really depend on $\beta$ but rather
on the coset $\beta+G$ (which can be chosen in two different
ways, namely as $\beta+G$ or $-\beta+G$).

Let $X$ be a simple weight $\mathfrak{a}_G$-module. Setting
$\mathfrak{n}^+_G X=0$ we turn $X$ into a
$\mathfrak{a}_G\oplus \mathfrak{n}^+_G$-module. We define
the {\em generalized Verma module} $M(G,\beta,X)$ as follows:
\begin{displaymath}
M(G,\beta,X):=
U(\mathfrak{g})\bigotimes_{\mathfrak{a}_G\oplus \mathfrak{n}^+_G}X.
\end{displaymath}
The module $M(G,\beta,X)$ is an indecomposable weight module and it
has a unique simple quotient, which we will denote by
$L(G,\beta,X)$. In \cite{BZ} it was shown that $L(G,\beta,X)$ is a
Harish-Chandra module if $X$ is a uniformly bounded exp-polynomial
module. Moreover, the module $L(G,\beta,X)$ itself is not uniformly
bounded unless $X$ is the trivial module (in which case
$L(G,\beta,X)$ is the trivial module itself). Loosely speaking we
will call $L(G,\beta,X)$ a {\em simple highest weight} module.

\section{Description of supports for Harish-Chandra
$\mathfrak{g}$-modules}\label{s2}

\subsection{Formulation of the main result}\label{s2.1}

Our main result is the following statement:

\begin{theorem}\label{mainthm}
Let $V$ be a nontrivial simple Harish-Chandra $\mathfrak{g}$-module. Then
exactly one of the following statements takes place:
\begin{enumerate}[(a)]
\item\label{mainthm.1} $V$ is dense and uniformly bounded.
\item\label{mainthm.2} $V$ is punctured and uniformly bounded.
\item\label{mainthm.3} $V\cong L(G,\beta,X)$ for some
$G$, $\beta$ and a uniformly bounded $X$ as in Subsection~\ref{s1.2}.
\end{enumerate}
\end{theorem}

For $n=1$ all Harish-Chandra modules over $\mathfrak{W}_1$ were
classified in \cite{M}. In this case the statement of
Theorem~\ref{mainthm} follows immediately from this
classification. Dense and punctured $\mathfrak{W}_1$-module occur
as modules from the so-called {\em intermediate series}. In
particular, in what follows we assume $n>1$.

For $n>1$ the existence of modules of the type
Theorem~\ref{mainthm}\eqref{mainthm.3} follows from \cite{BZ}.
Examples of both dense and punctured $\mathfrak{W}_1$-modules
were constructed in \cite{Ra} (and in a slightly different
setting already in \cite{Sh}).

For higher rank Virasoro algebras a complete classification of all
simple Harish-Chandra modules was obtained in \cite{LZ}. From this
classification it follows that the statement of
Theorem~\ref{mainthm} is true for all higher rank Virasoro
algebras (simple highest weight modules over higher rank Virasoro
algebras are defined exactly in the same way as in
Subsection~\ref{s1.2}, we will denote these modules by
$L^{\mathfrak{a}}(G,\beta,X)$ where $\mathfrak{a}$ is the higher
rank Virasoro algebra in question).

\subsection{Auxiliary lemmata for not uniformly
bounded modules}\label{s2.2}  To prove Theorem~\ref{mainthm} we
will need several auxiliary statements. In the whole section we
assume that $V$ is a simple Harish-Chandra $\mathfrak{g}$-module
and that for some (usually fixed) $\lambda\in\mathfrak{g}_0^*$ we
have $V=\oplus_{\alpha\in\Z^n}M_{\lambda+\alpha}$ .

To start with, we assume that the module $V$ is not
uniformly bo\-un\-ded. Let $e_i=(\delta_{i1},\delta_{i2}\dots,\delta_{in})$,
where $\delta_{ij}$ is the Kronecker delta. Then
$\{e_i:i=1,2,\dots,n\}$ is the standard basis of $\mathbb{Z}^n$.

\begin{lemma}\label{l1}
Assume that  $V$ is not uniformly bounded. Then, after an appropriate
change of variables $t_1,t_2,\dots,t_n$ and the weight
$\lambda$, we may assume that  $\lambda\neq 0$ and there exists a nonzero
$v_0\in V_{\lambda}$ such that
\begin{equation}\label{gh1}
\mathfrak {g}_{e_i} v_0=0,\,\text{ for all }\, i=1,2,\dots,n.
\end{equation}
\end{lemma}

\begin{proof}
Choose some $\gamma\in\Bbbk^n$ such that
$\gamma_1,\gamma_2,\dots,\gamma_n$ are linearly independent over
$\mathbb{Q}$ and consider $V$ as a Harish-Chandra
$\mathrm{Vir}(\gamma)$-module, which is not uniformly bounded.

From \cite[Theorem~3.9]{LZ} and \cite{BZ} we obtain that every
uniformly bounded  $\mathrm{Vir}(\gamma)$-module is either dense
or punctured. Hence the total number of uniformly bounded simple
subquotients of the $\mathrm{Vir}(\gamma)$-module $V$ is finite
(it is bounded by $\dim_{\Bbbk} V_{0}+\dim_{\Bbbk} V_{\mu}<\infty$
for any nonzero $\mu\in\lambda+\mathbb{Z}^n$). As the module $V$
itself is not uniformly bounded, it must contain a simple
$\mathrm{Vir}(\gamma)$-subquotient, say $X$, which is not
uniformly bounded.  By \cite[Theorem~3.9]{LZ}, the module $X$ is
a highest weight $\mathrm{Vir}(\gamma)$-module and, after an
appropriate change of variables $t_1,t_2,\dots,t_n$ and the weight
$\lambda$, the module $X$ is isomorphic to the module
$L^{\mathrm{Vir}(\gamma)}(G,e_1,Y)$, where $G=\mathbb{Z}\langle
e_2,e_3,\dots,e_n\rangle$ and $Y$ is the corresponding simple
uniformly bounded $\mathfrak{b}_G$-module (here
$\mathfrak{b}_G=\mathfrak{a}_G\cap \mathrm{Vir}(\gamma)$). By
\cite{BZ} we have that the dimensions $\dim_{\Bbbk}
V_{-ke_1+\lambda}$, $k\in\mathbb{N}$, are not uniformly bounded.

Fix an integer $N>3$ and consider the finite set
\begin{displaymath}
B_N(\lambda)=\lambda+\{\a\in\Z^n\,:\,|\a_i|\leq N
\text{ for all }i\}.
\end{displaymath}
As $V$ is a Harish-Chandra module and $\dim_{\Bbbk} V_{-ke_1+\lambda}$,
$k\in\mathbb{N}$, are not uniformly bounded,
there exists a positive integer $k$ such that we have
$-ke_1+\lambda\ne 0$ and
\begin{equation}\label{eqno3}
\dim_{\Bbbk} V_{-ke_1+\lambda}>
n\sum _{\beta\in B_N(\lambda)}\dim_{\Bbbk} V_{\beta}.
\end{equation}

Set $e_1'=(k+1)e_1+e_2$, $e_2'=ke_1+e_2$, and, finally,
$e_j'= e_1'+e_j$ for all
$j=3,4,\dots,n$. Then  $\{e'_1,e'_2,...,e_n'\}$ is a new
$\mathbb{Z}$-basis of $\mathbb{Z}^n$.
Moreover, for any $i=1,2,\dots,n$ we have
$e'_i+(-ke_1+\lambda)\in B_N(\lambda)$. Note that
$\dim_{\Bbbk}\mathfrak{g}_{e'_i}=n$ (observe the factor $n$ in the
right hand side of \eqref{eqno3}). Hence,
because of our choice of $k$ above, there exists a nonzero
$v_0\in  V_{-ke_1+\lambda}$ such that $\mathfrak{g}_{e'_i}v_0=0$
for all $i$. Thus, after another change of variables
(corresponding to the choice of the $\mathbb{Z}$-basis
$\{e'_i\}$ of $\mathbb{Z}^n$) and replacement of
$\lambda$ with $-ke_1+\lambda$, we obtain $v_0\in V_{\lambda}$
such that $\mathfrak {g}_{e_i} v_0=0$ for all $i=1,2,\dots,n$.
This completes the proof of the lemma.
\end{proof}

To proceed we need some more notation.
For any $\alpha,\beta\in\Z^n$, we write $\alpha>\beta$ and
$\alpha\geq \beta$ if $\alpha_i>\beta_i$ or $\alpha_i\geq \beta_i$ for
$i=1,2,\dots,n$, respectively. For $p,q\in \mathbb{Z}$
we set $[p,q]=\{x \in \Z\,:\, p \leq x
\leq q\}$ and define $(-\infty,q]$ and $[p,+\infty)$ similarly.

An element $v\in V$ will be called a {\em generalized highest
weight} element provided that there exists some $N\in\mathbb{N}$
such that $\mathfrak{g}_{\alpha}v=0$ for every
$\alpha\in\mathbb{Z}^n$ such that $\alpha>(N,N,\dots,N)$.
Analogues for different algebras of the next two claims appeared
in various disguises and setups e.g. in \cite{Ma1,Ma2,Su,LZ}.

\begin{lemma}\label{lem101}
Let $X$ be a weight $\mathfrak{g}$-module and $x\in X$ be a
generalized highest weight element. Then every element in
$U(\mathfrak{g})x$ is a generalized highest weight element. In
particular, if $V$ is as in Lemma~\ref{l1} satisfying \eqref{gh1},
then every $v\in V$ is a generalized highest weight element.
\end{lemma}

\begin{proof}
Since the algebra $U(\mathfrak{g})$ is generated by all
$\mathfrak{g}_{\alpha}$, to prove the first assertion it is enough to
show that if $x$ is a generalized highest weight element,
$\beta\in \mathbb{Z}^n$ and $a\in \mathfrak{g}_{\beta}$, then
the element $ax$ is a generalized highest weight element.

Assume that $N\in\mathbb{N}$ is such that
$\mathfrak{g}_{\alpha}v=0$ for every $\alpha\in\mathbb{Z}^n$ with
$\alpha>(N,N,\dots,N)$. Set
$N'=N+|\beta_1|+|\beta_2|+\cdots+|\beta_n|+1$. Then for every
$\alpha>(N',N',\dots,N')$ we have
\begin{displaymath}
\mathfrak{g}_{\alpha}\mathfrak{g}_{\beta}v\subset
\mathfrak{g}_{\beta}\mathfrak{g}_{\alpha}v+
\mathfrak{g}_{\beta+\alpha}v=0
\end{displaymath}
as $\alpha,\beta+\alpha>(N,N,\dots,N)$ by our choice of $N'$.
The first assertion of the lemma follows.

The module $V$ from Lemma~\ref{l1} is simple and hence is generated
by every nonzero element. By \eqref{gh1}, the element $v_0$ is a
generalized highest weight element. Therefore the second assertion of our
lemma follows directly from the first assertion.
\end{proof}

\begin{lemma}\label{lem107}
Let $V$ be as in Lemma~\ref{l1} satisfying \eqref{gh1}.
Then $\mathfrak{g}_{-\alpha}v\neq 0$ for
any nonzero $v\neq V$ and any $\alpha\in\mathbb{N}^n$.
\end{lemma}

\begin{proof}
Assume that $\mathfrak{g}_{-\alpha}v=0$ for some nonzero $v\neq
V$. By Lemma~\ref{lem101}, there exists $N\in \mathbb{N}$ such
that $\mathfrak{g}_{e_i+N\alpha}v=0$ for every $i=1,2,\dots,n$. It
is easy to see that the monoid $\mathbb{Z}^n$ is generated, as a
monoid, by the elements $e_i+N\alpha$, $i=1,2,\dots,n$, and
$-\alpha$. It follows that $\mathfrak{W}_n$ is generated, as a Lie
algebra, by $\mathfrak{g}_{e_i+N\alpha}$, $i=1,2,\dots,n$, and
$\mathfrak{g}_{-\alpha}$. Hence $\mathfrak{W}_nv=0$, which implies
that $V$ is the trivial module. This contradicts our assumption
that $V$ is not uniformly bounded and the claim of the lemma
follows.
\end{proof}

Analogues of the next two claims appeared in the
setup of generalized Virasoro algebras in \cite[Lemma~3.1]{LZ}.
Our proofs below are generalizations of the ones from \cite{LZ}.

\begin{lemma}\label{lem102}
Let $V$ be as in Lemma~\ref{l1} satisfying \eqref{gh1}.
Then for any $\mu\in\supp (V)$ and any $\alpha\in\mathbb{N}^n$
we have
\begin{displaymath}
\{x \in \mathbb{Z}\,:\,\mu+x\alpha \in \supp(V)\}=(-\infty, m]
\end{displaymath}
for some $m\in \mathbb{N}\cup\{0\}$.
\end{lemma}

\begin{proof}
Set $J:=\{x \in \mathbb{Z}\,:\,\mu+x\alpha \in \supp(V)\}$.
From Lemma~\ref{lem107} we have that either
$J=(-\infty,m]$ for some $m\in \mathbb{Z}$ or $J=\mathbb{Z}$.
Suppose that $J=\mathbb{Z}$. Choose some $\partial\in\mathfrak {g}_0$
such that $\partial (t^\beta)=0$ implies $\beta=0$ for all $\beta\in\mathbb{Z}^n$.

Consider the subalgebra
$\mathfrak{V}_{\alpha}$ of $\mathfrak{g}$, generated by
the elements $t^{k\alpha}\partial$, $k\in\mathbb{Z}$. This
subalgebra is a classical centerless Virasoro algebra (of rank one)
and the space $X_{\alpha}:=\oplus_{x\in\Z} V_{\mu+x\alpha}$
admits a natural structure of a $\mathfrak{V}_{\alpha}$-module,
given by restriction.

From Lemma~\ref{lem101} we obtain that for any
$v\in X_{\alpha}$ we have $t^{k\alpha}\partial(v)=0$ for all
$k\in\mathbb{N}$ big enough. Hence, by \cite[Lemma~1.6]{M}, for every
$m\in\mathbb{Z}$ there exists $m'\in \mathbb{Z}$, $m'>m$,
and a nonzero vector $v(m')\in V_{\mu+m'\alpha}$ annihilated by
$t^{k\alpha}\partial$ for all $k\in \mathbb{N}$.

Therefore the weight $\mu$ occurs in infinitely many simple
highest weight $\mathfrak{V}_{\alpha}$-subquotients of
$X_{\alpha}$, implying $\dim_{\Bbbk}V_{\mu}=\infty$.
This contradicts our assumption that $V$ is a Harish-Chandra
module. Thus $J\neq \mathbb{Z}$ and the claim of the lemma follows.
\end{proof}

\begin{lemma}\label{l3}
Let $V$ be as in Lemma~\ref{l1} satisfying \eqref{gh1}.
Then one can change the variables $t_1,t_2,\dots,t_n$
(keeping the weight $\lambda$) such that
there exists $v_0\in V_{\lambda}$
with the following properties:
\begin{enumerate}[(a)]
\item\label{l3.1} The condition \eqref{gh1} is satisfied.
\item\label{l3.2} $\lambda+\alpha\not\in\mathrm{supp}(V)$ for any
nonzero $\alpha\in\mathbb{Z}_+^n$. \item\label{l3.3}
$\lambda-\alpha\in\mathrm{supp}(V)$ for any
$\alpha\in\mathbb{Z}_+^n$. \item\label{l3.4} for any
$\alpha,\beta\in \mathbb{Z}^n$ such that $\alpha\leq \beta$ we
have $\lambda+\alpha\not\in\mathrm{supp}(V)$  implies
$\lambda+\beta\not\in\mathrm{supp}(V)$.
\end{enumerate}
\end{lemma}

\begin{proof}
By Lemma~\ref{lem102} we have $\{x \in\mathbb{Z}\,:
\,\lambda+x(1,1\cdots,1) \in\mathrm{supp} (V)\}=(\infty, p-2]$
for some $p\geq 2$. Take $e'_1=(p+1,p,\dots,p)$,
$e'_2=(p+2,p+1,p,\dots,p)$
and $e'_i=e'_1+e_i$ for all $i=3,4,\dots,n$. This gives us a new
$\mathbb{Z}$-basis of $\mathbb{Z}^n$. The condition
\eqref{l3.1} is obvious. The condition \eqref{l3.4} follows from
Lemma~\ref{lem102}. The conditions
\eqref{l3.2} and \eqref{l3.3} are proved similarly to the
proof of Lemma~\ref{lem107} (note that $\lambda+(p-1)(1,1,\dots,1)\not
\in\mathrm{supp}(V)$).
\end{proof}

\subsection{The key lemma}\label{s2.4}
The following statement is our key observation.
For $a,b\in\mathbb{Z}^n$, we set $a\cdot b=a_1b_1+a_1b_1+...+a_1b_1$.
For any subgroup $G$ of $\Z^n$ and any $\mu\in\mathfrak{g}_0^*$ the space
$\displaystyle V(\mu+G)=\bigoplus_{\alpha\in \mu+G}V_{\alpha}$ is
naturally a $\mathfrak{g}(G)$-module by restriction.

\begin{lemma}\label{l4}
Assume that $V$ is not uniformly bounded. Then $V$ is a highest weight
module as in Theorem~\ref{mainthm}\eqref{mainthm.3}.
\end{lemma}

\begin{proof}
By Lemma~\ref{l3} we may assume that $V$ satisfies
Lemma~\ref{l3}\eqref{l3.1}--\eqref{l3.4}.

Fix $\gamma\in\Bbbk^n$ such that $\gamma_1, \gamma_2,\cdots, \gamma_n$
are linearly independent over $\mathbb{Q}$ and consider
$V$ as a Harish-Chandra $\mathrm{Vir}(\gamma)$-module by restriction.
To simplify our notation, set $\mathfrak{a}:=\mathrm{Vir}(\gamma)$.
Let $Y$ be a minimal $\mathfrak{a}$-submodule of $V$ such
that $Y\cap V_\lambda\neq 0$, and $Z$  the maximal
$\mathfrak{a}$-submodule of $Y$ such that $Z\cap V_\lambda=0$.
Then the  $\mathfrak{a}$-module $Y/Z$ is simple.

As both $\lambda\neq 0$ and $V_\lambda\neq 0$ by Lemma~\ref{l3},
from \cite[Theorem~3.9]{LZ} it follows that the
$\mathfrak{a}$-module $Y/Z$ is isomorphic to
$L^{\mathfrak{a}}(G,\beta,X)$ for some $G$, $\beta$ and $X$ as
described in Subsection~\ref{s1.2} (but for the algebra
$\mathfrak{a}$ and after the identification of $\lambda$ with
$\lambda(\gamma_1\partial_1+\gamma_2\partial_2+\cdots+\gamma_n\partial_n)$).
Moreover, $X$ is uniformly bounded (and hence is a module from the
intermediate series). It now follows that
\begin{equation}\label{eqno5}
\left(\lambda-\mathbb{Z}_+ \beta + G\right)\setminus\{0\}
\subset \supp(V).
\end{equation}
Moreover, from  Lemma~\ref{l3}\eqref{l3.2} we have $\lambda+\alpha
+ G\not\subset \supp(V)$ for any nonzero
$\alpha\in\mathbb{Z}_+^n$.

From  Lemma~\ref{l3}\eqref{l3.2} it follows that there exists
$\alpha\in \mathbb{N}^n$ such that we have
$G=\{x\in\mathbb{Z}^n\,:\, \alpha\cdot x=0\}$. (For example, if
$\alpha_1=0$, then $le_1\in G$ for all $l\in\Z$, and $\lambda
+l_1e_1\notin \mathrm{supp}(V)$ for sufficiently large $l_1$ which
contradicts the fact that $(\lambda+ G)\setminus\{0\}\subset
\supp(V)$.) It follows from \eqref{eqno5} that $\lambda+x\in
\supp(V)$ for all $x\in\mathbb{Z}^n$ with $x\cdot\alpha<0$.

We first consider the case when
$\{\lambda+k\beta+G\}\cap\mathrm{supp}(V)=\varnothing$
for some $k\in\mathbb{N}$. We may assume that $k$ is minimal possible.
In this case for any $\mu\in \{\lambda+(k-1)\beta+G\}\cap\mathrm{supp}(V)$
(note that the latter intersection is not empty because of our
assumption on $k$)
and any $x\in G$ we have $\mathfrak{g}_{x+\beta}V_{\mu}=0$. Hence
$V\cong L(G,\beta,X')$, where
\begin{displaymath}
X'=\bigoplus_{\mu\in \{\lambda+(k-1)\beta+G\}}V_{\mu}.
\end{displaymath}

Now consider the remaining case when
$\{\lambda+k\beta+G\}\cap\mathrm{supp}(V)\neq \varnothing$
for all $k\in\mathbb{N}$. Obviously, we can choose $k\in\mathbb{N}$
big enough such that the following two conditions are satisfied:
\begin{enumerate}[(I)]
\item\label{cond1}
$|\{\lambda+k\beta+G\}\cap \{\lambda+\mathbb{Z}_+^n\}|>1$
\item\label{cond2}
$|\{\lambda+(k-1)\beta+G\}\cap \{\lambda+\mathbb{Z}_+^n\}|>0$.
\end{enumerate}

The space $V(\lambda+k\beta+G)$ is a Harish-Chandra
$\mathfrak{g}(G)$-module. Restricting this module to any
generalized Virasoro subalgebra of $\mathfrak{g}(G)$ (of the same
rank) and using \eqref{cond1}, \cite{LZ} and \cite{BZ} in the same
way as we did in the proof of Lemma~\ref{l1}, we get that this
module is not uniformly bounded. Hence we can repeat the arguments
of Lemma~\ref{l1} and find a $\mathbb{Z}$-basis
$\beta_2,\beta_3,\dots,\beta_n$ of $G$, $\mu\in \lambda+k\beta+G$,
and a nonzero element $v\in V_{\mu}$, such that
$\mathfrak{g}_{\beta_i}v=0$ for all $i=2,3,\dots,n$.

Let $\nu\in \{\lambda+(k-1)\beta+G\}\cap \{\lambda+\mathbb{Z}_+^n\}$,
which exists by \eqref{cond2}. Then $\nu\not\in\mathrm{supp}(V)$
by Lemma~\ref{l3}\eqref{l3.2}. Let $\beta'_1=\nu-\mu$. Then
$\beta'_1,\beta_2,\beta_3,\dots,\beta_n$ is a $\mathbb{Z}$-basis of
$\mathbb{Z}^n$ and $\mathfrak{g}_{\beta'_1}v=0$ as well.

From Lemma~\ref{lem102} we obtain that
$\mu+m(\beta'_1+\beta_2+...+\beta_n)\not\in\mathrm{supp}(V)$ for
all sufficiently large integers $m$. At the same time for
$m\in\mathbb{N}$ we have
\begin{displaymath}
\alpha\cdot m(\beta'_1+\beta_2+...+\beta_n)=
m\alpha\cdot \beta'_1=-m\alpha\cdot \beta<0.
\end{displaymath}
Hence for all $m$ sufficiently large we have
\begin{displaymath}
\mu+ m(\beta'_1+\beta_2+...+\beta_n)
\in \lambda-\mathbb{Z}_+ \beta + G.
\end{displaymath}
This contradicts \eqref{eqno5}. Hence it is not possible that the
intersection
$\{\lambda+k\beta+G\}\cap\mathrm{supp}(V)$ is nonempty
for all $k\in\mathbb{N}$. The claim of the lemma follows.
\end{proof}

\subsection{Proof of Theorem~\ref{mainthm}}\label{s2.5}

From Lemma~\ref{lem102} it follows that every simple dense and
every simple punctured $\mathfrak{g}$-module is uniformly bounded.
Let now $V$ be a simple nontrivial Harish-Chandra
$\mathfrak{g}$-module, which is neither dense nor punctured. Let
$\gamma\in\Bbbk^n$ be such that $\gamma_1,\gamma_2,\dots,\gamma_n$
are linearly independent over $\mathbb{Q}$. Then $V$ is a
$\mathrm{Vir}(\gamma)$-module by restriction. As $V$ is neither
dense nor punctured, from \cite[Theorem~3.9]{LZ} it follows that
every simple nontrivial subquotient of $V$ is a highest weight
module. In particular, $V$ is not uniformly bounded by \cite{BZ}.
From Lemma~\ref{l4} we now get that $V$ is a highest weight
$\mathfrak{g}$-module as described in
Theorem~\ref{mainthm}\eqref{mainthm.3}. This completes the proof.

\section{Support of non Harish-Chandra weight modules}\label{s3}

In this section we would like to prove some analogue of
Theorem~\ref{mainthm} for all weight modules (that is without the
assumption to be a Harish-Chandra module). At this moment, we
cannot get such a nice statement as the one in
Theorem~\ref{mainthm}, but we  get some information about the
support of the module generalizing the corresponding result for
the Virasoro algebra (\cite[Theorem~2]{Ma3}).

Actually in this section our algebra $\mathfrak {g}$ can be more
general than $W_n$. We assume that $\mathfrak {g}$ has a
$\Z^n$-gradation   $\displaystyle\mathfrak {g}=
\bigoplus_{\alpha\in\Z^n}\mathfrak
{g}_\alpha$ such that $\mathfrak {g}_0$ is abelian and the gradation
itself is  the root space decomposition with respect to $\mathfrak {g}_0$.
We also assume that $[\mathfrak {g}_\alpha,\mathfrak {g}_\beta]=\mathfrak
{g}_{\alpha+\beta}$ for all $\alpha,\beta\in\Z^n$ with $\alpha\ne
\beta$. Clearly both $W_n$ and higher rank Virasoro algebras are 
examples of such Lie algebras.

\subsection{Cut modules}\label{s3.1}

For $a\in\mathbb{R}^n$ set
\begin{gather*}
\mathbb{Z}^{(a)}_-=\{x\in \mathbb{Z}^n:a\cdot x<0\};\\
\mathbb{Z}^{(a)}_+=\{x\in \mathbb{Z}^n:a\cdot x>0\};\\
\mathbb{Z}^{(a)}_0=\{x\in \mathbb{Z}^n:a\cdot x=0\};\\
\mathbb{Z}^{(a)}_{-0}=\{x\in \mathbb{Z}^n:a\cdot x\leq 0\};\\
\end{gather*}
Following \cite{Ma1,Ma2} we call a simple weight $\mathfrak{g}$-module
$V$ {\em cut} provided that there exists $\lambda\in\mathrm{supp}(V)$,
$a\in\mathbb{R}^n$, $a\neq 0$, and $b\in \mathbb{Z}^n$ such that
$\mathrm{supp}(V)\subset \lambda+b+\mathbb{Z}^{(a)}_{-0}$.

Let $L(G,\beta,X)$ be a simple highest weight module as in
Subsection~\ref{s1.2}. It is easy to see that there exists
$a\in\mathbb{R}^n$, $a\neq 0$, such that $G\subset
\mathbb{Z}^{(a)}_0$ and $\beta\in \mathbb{Z}^{(a)}_+$ (actually in
this case we can even take $a\in\Z^n$).  For any
$\lambda\in\mathrm{supp}(X)$ we have
$\mathrm{supp}(L(G,\beta,X))\subset\lambda+\mathbb{Z}^{(a)}_{-0}$,
in particular, $L(G,\beta,X)$ is a cut module (where we take $b=0$).

In the general case one cannot get rid of the element $b$ in the
definition of cut modules. Choose some $a\in\mathbb{R}^n$, $a\neq
0$, such that $\mathbb{Z}^{(a)}_0=\{0\}$. Set
\begin{displaymath}
\mathfrak{n}^{\pm}=\bigoplus_{\alpha\in \mathbb{Z}^{(a)}_{\pm}}
\mathfrak{g}_{\alpha}.
\end{displaymath}
We make $\Bbbk$ as the trivial $\mathfrak{b}=
\mathfrak{g}_{0}\oplus\mathfrak{n}^{+}$-module. Then the kernel of
the canonical epimorphism from the Verma module
$U(\mathfrak{g})\otimes_{U(\mathfrak{b})}\Bbbk$ onto the trivial
$\mathfrak{g}$-module contains a weight irreducible subquotient
with support $\mathbb{Z}^{(a)}_-$ since $\mathfrak{g}$ contains a
rank $n$ Virasoro algebra(using \cite{HWZ}).

Our main result of this subsection is the following:

\begin{theorem}\label{thm31}
Let $V$ be a simple weight $\mathfrak{g}$-module, which is neither dense
nor trivial. Assume that $V$ contains a generalized highest weight element
for some choice of variables $t_1,\dots,t_n$. Then $V$ is a cut module.
\end{theorem}

\begin{proof}
For $n=1$ this is proved in  \cite[Theorem~2]{Ma3}, so in what follows
we assume $n>1$. By Lemma~\ref{lem101}, every element of $V$ is a
generalized highest weight element with respect to our fixed
choice of $t_1,\dots,t_n$. We will use real convexity theory in our
arguments, so we will need to fix the corresponding setup.

Fix $\lambda\in \mathrm{supp}(V)$, $\lambda\neq 0$, and consider the set
$\Lambda=\lambda+\mathbb{Z}^n$, which we identify with $\mathbb{Z}^n$
and consider as a subset of $\mathbb{R}^n$. Set further
$\hat{\Lambda}=\Lambda\setminus \mathrm{supp}(V)$ and note that
$\hat{\Lambda}\neq \varnothing$ because of our assumption that
$V$ is not dense. The key point of the proof is the following
observation, which establishes a convexity property for $\hat{\Lambda}$:

\begin{lemma}\label{lem34}
Let $\mu\in \Lambda$ be a convex linear combination of some elements
from $\hat{\Lambda}$. Then either $\mu=0$ or $\mu\in \hat{\Lambda}$.
\end{lemma}

\begin{proof}
Assume that
\begin{equation}\label{eq301}
\mu=\sum_{i=1}^k a_i\mu_i,
\end{equation}
where $\mu_i\in \hat{\Lambda}$, $a_i\in\mathbb{R}$, $a_i>0$, $k>1$
and $\sum_{i=1}^k a_i=1$. We may even assume that $k$ is minimal
possible. The latter says that $\mu$ belongs to the interior of the
convex hull $H$ of the $\mu_i$'s. By \cite[Corollary~2.7.2]{La} we
may assume that $\mu_i-\mu$, $i=1,\dots,k$, are affinely
independent.

Denote by $X$ the subspace of $\mathbb{R}^n$, generated by the
elements $\mu_i-\mu$, $i=1,\dots,k$. Then $H-\mu\subset X$ and
since $\mu$ is a point in the interior of $H$, we  get that the
convex cone in $X$ with origin in $\mu$, which contains all
$\mu_i-\mu$, $i=1,\dots,k$, coincides with $X$. By
\cite[Lemma~2.6.2]{La} we have $\dim X=k-1$ and hence without loss
of generality we may assume that the elements $\mu_i-\mu$,
$i=2,\dots,k$, are linearly independent. Then from \eqref{eq301}
we have
\begin{equation}\label{eq302}
-(\mu_1-\mu)=\sum_{j=2}^{k} \frac{a_i}{a_1}(\mu_i-\mu).
\end{equation}
As the vectors $\mu_i-\mu$, $i=2,\dots,k$, are linearly independent,
the equation \eqref{eq302} gives the unique linear combination of
these elements, which is equal to $-(\mu_1-\mu)$. Since all involved
elements $\mu_i, \mu$ are from $\mathbb{Z}^n$, it follows that all
$\frac{a_i}{a_1}$ in \eqref{eq302} are rational numbers (and are
positive). Multiplying, if necessary, with the denominator, we
obtain the equality
\begin{equation}\label{eq303}
-b_1(\mu_1-\mu)=\sum_{j=2}^{k} b_i(\mu_i-\mu),
\end{equation}
where all $b_i$'s are positive integers.

Note that for every $i=1,2,\dots,k$ we have
$\mathfrak{g}_{\mu_i-\mu}V_{\mu}\subset V_{\mu_i}=0$ by our
assumptions. From  \eqref{eq303} we see that $\mathfrak{g}_{0}$ is
in the Lie subalgebra generated by all $\mathfrak{g}_{\mu_i-\mu},
i=1,2,...,k$. Therefore $\mathfrak{g}_{0}V_{\mu}=0$. This implies $\mu=0$
or $V_{\mu}=0$ and completes the proof of our lemma.
\end{proof}

Let $\overline{\Lambda}$ denote the convex hull of $\hat{\Lambda}$
under the usual topology in $\mathbb{R}^n$. As the module $V$ is not
dense, we have $\overline{\Lambda}\neq \varnothing$. As the module
$V$ is not trivial, it must contain a nonzero weight and hence
$\overline{\Lambda}\neq \mathbb{R}^n$ by Lemma~\ref{lem34}. From the
hyperplane separation theorem (see e.g. \cite[3.2]{La}) it   follows
that for every $\mu\in\mathrm{supp}(V)\setminus\{0\}$ there exists
$a_{\mu}\in\mathbb{R}^n$, $a_{\mu}\neq 0$, such that
$\hat{\Lambda}\subset \mu+\mathbb{Z}_+^{(a_{\mu})}$. Thus for the
fix $\lambda$ there is a nonzero $a\in\mathbb{R}^n$ such that
$\hat{\Lambda}\subset \lambda+\mathbb{Z}_+^{(a)}$.

To proceed we will need a simple fact about ordered abelian groups.

\begin{lemma}\label{lem051}
Let $\beta\in \mathbb{Z}^n$ be such that $\beta\cdot a<0$. Then
there exists a finite collection  of elements
$\beta^{\pm}_1,\dots,\beta^{\pm}_{n}$ from $\mathbb{Z}_+^{(a)}$ such
that $\{\beta,\beta^{\pm}_1,\dots,\beta^{\pm}_n\}$ generates
$\mathbb{Z}^n$ as a semigroup.
\end{lemma}

\begin{proof}
As $\beta\cdot a<0$, for every $\alpha\in \mathbb{Z}^n$ we can fix
some $k_{\alpha}\in\{0,1,\dots\}$ such that
$\alpha-k_{\alpha}\beta\in \mathbb{Z}_+^{(a)}$.
For $i\in\{1,\dots,n\}$ set $\beta^{\pm}_i=\pm e_i-k_{\pm e_i}\beta$.
Then all elements $\pm e_i$ belong to the semigroup, generated by
$\{\beta,\beta^{\pm}_1,\dots,\beta^{\pm}_n\}$, by construction and
hence the latter set generates the whole $\mathbb{Z}^n$ as a semigroup.
\end{proof}

For every $\beta\in \mathbb{Z}^n$ with $\beta\cdot a<0$, we fix
$\beta^{\pm}_1,\dots,\beta^{\pm}_{n}$ given by Lemma~\ref{lem051}.
Now we can generalize our analysis from Section~\ref{s2}.

\begin{lemma}\label{lem033}
Let $\beta\in \mathbb{Z}^n$ be such that $\beta\cdot a<0$. Then
there exists $\mu\in \mathrm{supp}(V)$ such that
$\mu-\beta\in \hat{\Lambda}$.
\end{lemma}

\begin{proof}
Let $\nu\in \hat{\Lambda}\neq\varnothing$. Consider the ray
$\{\nu+k\beta:k\in\mathbb{N}\}$. As $\beta\cdot a<0$, this ray must
intersect $\lambda+\mathbb{Z}_-^{(a)}\subset \mathrm{supp}(V)$.
The claim follows.
\end{proof}

\begin{lemma}\label{lem034}
Let $v\in V$ be a nonzero weight vector. Assume that there exists
$\beta\in \mathbb{Z}^n$ such that $\beta\cdot a<0$ and
$\mathfrak{g}_{\beta}v=0$. Then for any $w\in V$ there exists $N\in
\mathbb{N}$ such that for any $\alpha=b\beta+\gamma$, where
$b\in\mathbb{N}$ and $\gamma\in\mathbb{N}^n$ with
$\gamma>(N,N,\dots,N)$, we have $\mathfrak{g}_{\alpha}w=0$.
\end{lemma}

\begin{proof}
From Lemma~\ref{lem101} we know that $v$ is a generalized highest
weight element and hence there exists $N_1\in \mathbb{N}$ such
that we have $\mathfrak{g}_{\gamma}v=0$ for all
$\gamma\in\mathbb{N}^n$, $\gamma>(N_1,N_1,\dots,N_1)$. If not all
components $\beta_i$ of $\beta$ are negative, using that
$[\mathfrak{g}_{x},\mathfrak{g}_{y}]= \mathfrak{g}_{x+y}$ if $x\ne
y$, in the case $w=v$ the claim of the lemma follows by induction
on $b$ and taking $N=N_1$. Now suppose all components $\beta_i$ of
$\beta$ are negative. In the case $w=v$ we take $N=2N_1+2$. For
$\gamma
>(N,N,\dots,N)$ with $0\notin \mathbb{N}\beta+\gamma$ the claim of
the lemma follows by induction on $b$. For $\gamma
>(N,N,\dots,N)$ with $\gamma=-k\beta, k>0$, we write $\gamma=\gamma_1+\gamma_2$ such that
$\gamma_1,\gamma_2>(N,N,\dots,N)$, $0\notin
\mathbb{N}\beta+\gamma_1$ and  $\gamma_2\ne b\beta+\gamma_1$. Then
we know that $\mathfrak{g}_{b\beta+\gamma_1}v=0$, further
$\mathfrak{g}_{b\beta+\gamma}v=0$. The claim in  the Lemma follows
for the case $w=v$.

Using the same arguments as in Lemma~\ref{lem101} one shows that if the
claim of our lemma is true for some $w$, it is true for any element
from $\mathfrak{g}_{x}w$, $x\in\mathbb{Z}^n$. The general claim of the
lemma now follows from the fact that $V$ is simple and hence generated
by $v$.
\end{proof}

\begin{lemma}\label{lem035}
Let $v\in V$ be a nonzero weight vector. Assume that there exists
$\beta\in \mathbb{Z}^n$ such that $\beta\cdot a<0$ and
$\mathfrak{g}_{\beta}v=0$. Then there exists $N\in \mathbb{N}$ such
that for any $\alpha=b\beta+\sum_{\varepsilon\in\{\pm\}}\sum_{i=1}^n
b^{\varepsilon}_i\beta^{\varepsilon}_i+\gamma$, where
$b,b^{\pm}_i\in\mathbb{N}$, $i=1,\dots,n$, and
$\gamma\in\mathbb{N}^n$ with $\gamma>(N,N,\dots,N)$, we have
$\mathfrak{g}_{\alpha}v=0$.
\end{lemma}

\begin{proof}
By Lemma~\ref{lem033} for every $\varepsilon\in\{\pm\}$ and
$i\in\{1,2,\dots,n\}$ there is
$\mu_{i}^{\varepsilon}\in\mathrm{supp}(V)$ such that
$\mu_{i}^{\varepsilon}+\beta^{\varepsilon}_i\in \hat{\Lambda}$.
This yields
$\mathfrak{g}_{\beta^{\varepsilon}_i}V_{\mu_{i}^{\varepsilon}}=0$.
Hence, by Lemma~\ref{lem034}, there exists
$N_{i}^{\varepsilon}\in\mathbb{N}$ such that
\begin{equation}\label{eq50501}
\mathfrak{g}_{b_{i}^{\varepsilon}\beta^{\varepsilon}_i+
\gamma_{i}^{\varepsilon}}v=0
\end{equation}
for any  $b_{i}^{\varepsilon}\in \mathbb{N}$ and any
$\gamma_{i}^{\varepsilon}\in\mathbb{N}^n$, $\gamma_{i}^{\varepsilon}>
(N_{i}^{\varepsilon},N_{i}^{\varepsilon},\dots,N_{i}^{\varepsilon})$.

Similarly, there exists $N_{\beta}\in\mathbb{N}$
such that
\begin{equation}\label{eq50502}
\mathfrak{g}_{b\beta+\gamma'}v=0
\end{equation}
for any $b\in\mathbb{N}$ and any $\gamma'\in\mathbb{N}^n$,
$\gamma'>(N_{\beta},N_{\beta},\dots,N_{\beta})$. Let
\begin{displaymath}
N= N_{\beta}+\sum_{\varepsilon\in\{\pm\}}
\sum_{i=1}^n N_{i}^{\varepsilon} +2n+1.
\end{displaymath}
Then every $\gamma>(N,N,\dots,N)$ can be written as
\begin{displaymath}
\gamma=\gamma'+ \sum_{\varepsilon\in\{\pm\}}
\sum_{i=1}^n \gamma_{i}^{\varepsilon},
\end{displaymath}
where $\gamma'>(N_{\beta},N_{\beta},\dots,N_{\beta})$ and
$\gamma_{i}^{\varepsilon}>
(N_{i}^{\varepsilon},N_{i}^{\varepsilon},\dots,N_{i}^{\varepsilon})$.
For this choice of $N$ using the expression
$$b\beta+\sum_{\varepsilon\in\{\pm\}}\sum_{i=1}^n
b^{\varepsilon}_i\beta^{\varepsilon}_i+\gamma=(b\beta+\gamma')+\sum_{\varepsilon\in\{\pm\}}\sum_{i=1}^n
(b^{\varepsilon}_i\beta^{\varepsilon}_i+\gamma^{\varepsilon}_i),$$
we deduce the claim of the lemma from \eqref{eq50501} and
\eqref{eq50502}.
\end{proof}

\begin{lemma}\label{lem036}
Let $\mu\in \hat{\Lambda}$ and $\beta\in\mathbb{Z}^n$ be such that
$\beta\cdot a<0$. Then $\mu-\beta\in \hat{\Lambda}.$
\end{lemma}

\begin{proof}
We have $\mathfrak{g}_{\beta}V_{\mu-\beta}\subset V_{\mu}=0$.
Suppose $V_{\mu-\beta}\ne0$ and fix any nonzero $v\in
V_{\mu-\beta}$. By Lemma~\ref{lem035},
 there exists $N\in \mathbb{N}$
such that for any $\alpha=b\beta+\sum_{i=1}^n
b^{\pm}_i\beta^{\pm}_i+\gamma$, where $b,b^{\pm}_i\in\mathbb{N}$,
$i=1,\dots,n$, and $\gamma\in\mathbb{N}^n$ with
$\gamma>(N,N,\dots,N)$, we have $\mathfrak{g}_{\alpha}v=0$.

Since $\{\beta,\beta^{\pm}_1,\dots,\beta^{\pm}_n\}$ generate $\mathbb{Z}^n$
as a semigroup (Lemma~\ref{lem051}), for any $\alpha\in \mathbb{Z}^n$
we can write
\begin{displaymath}
\alpha-(N+1,N+1,\dots,N+1)=b\beta+ \sum_{\varepsilon\in\{\pm\}}
\sum_{i=1}^n b_{i}^{\varepsilon}\beta^{\pm}_i
\end{displaymath}
for some $b,b^{\pm}_i\in\mathbb{N}$ and hence
\begin{displaymath}
\alpha=b\beta+ \sum_{\varepsilon\in\{\pm\}}
\sum_{i=1}^n b_{i}^{\varepsilon}\beta^{\pm}_i +(N+1,N+1,\dots,N+1).
\end{displaymath}
Note that $(N+1,N+1,\dots,N+1)>(N,N,\dots,N)$, which yields
$\mathfrak{g}_{\alpha}v=0$ from the previous paragraph. Thus $V$
is a trivial module which is a contradiction. Therefore
$V_{\mu-\beta}=0$ which completes the proof.
\end{proof}

Theorem~\ref{thm31} follows directly from Lemma~\ref{lem036}.
\end{proof}

\subsection{Case $n=2$}\label{s3.2}

In the case $n=2$ Theorem~\ref{thm31} implies the following
trichotomy result for all weight modules (see \cite{Ma2}):

\begin{corollary}\label{cor071}
Assume that $n=2$ and $V$ is a simple weight $\mathfrak{g}$-module.
Then $V$ is either dense or punctured or cut.
\end{corollary}

\begin{proof}
We work with the same setup (for $\Lambda$ and $\hat{\Lambda}$)
as in Theorem~\ref{thm31}.
Note that the trivial $\mathfrak{g}$-module is obviously cut. Hence,
because of Theorem~\ref{thm31}, it is enough to show  that any simple
nontrivial weight module $V$ for which $|\hat{\Lambda}|>1$ contains a
generalized highest weight element.

 Now suppose $|\hat{\Lambda}|>1$. If $\hat{\Lambda}$ contains a line in $\Z^n$, we can easily see
that $V$ is cut  (we are using the identification defined before
Lemma 9). Now we also suppose that each line in $\Z^n$ has at least
one weight of $V$. By changing the coordinate system $\{e_1, e_2\}$
and changing $\lambda$ (which is allowed to be $0$) if necessary, we
may assume that $(0,0)\notin \hat{\Lambda}$ and $(1,0), (1,k)\in
\hat{\Lambda}$
 for some
$k\in\mathbb{N}$. Using Lemma~\ref{lem34} we may even assume
$k\in\{1,2\}$. If $(1,1)\in\mathrm{supp}(V)$ we must have
$\lambda+e_1+e_2=0$, and further $\hat{\Lambda}=\{(1,0),(1,2)\}$
(otherwise we use other points instead of $ (1,0),(1,2)$ to get
$k=1$).

If $(1,1)\in \hat{\Lambda}$, then $\mathfrak{g}_{e_1}V_{\lambda}=
\mathfrak{g}_{e_1+e_2}V_{\lambda}=0$ and since $\{e_1,e_{1}+e_{2}\}$
is a $\mathbb{Z}$-basis of $\mathbb{Z}^2$ we obtain that any element
in $V_{\lambda}$ is  a generalized highest weight element.

Finally, we are left with the case $\hat{\Lambda}=\{(1,0),(1,2)\}$
(and we also have the equality $\lambda+e_1+e_2=0$, which  will not
be used). Consider the $\mathbb{Z}$-basis $\alpha=e_1$,
$\beta=2e_1+e_2$ of $\mathbb{Z}^2$. We claim that in this case for
any $a,b\in\mathbb{N}$ such that $a,b>1$ we have
$\mathfrak{g}_{a\alpha+b\beta}V_{\lambda}=0$ (and hence any element
in $V_{\lambda}$ is a generalized highest weight element with
respect to the basis $\{\alpha+\beta, \alpha+2\beta\}$). As
$\mathfrak{g}_{\alpha}V_{\lambda}=0$ by our assumptions, by
induction it is enough to show that
$\mathfrak{g}_{b\beta}V_{\lambda}=0$ for all $b\in\mathbb{N}$,
$b>1$.

If $b$ is even we have $b\beta=\frac{b}{2}(e_1+2e_2)+\frac{3b}{2}e_1$
and $\mathfrak{g}_{b\beta}V_{\lambda}=0$ follows from
$\mathfrak{g}_{e_1+2e_2}V_{\lambda}\subset V_{\lambda+e_1+2e_2}=0$
and $\mathfrak{g}_{e_1}V_{\lambda}\subset V_{\lambda+e_1}=0$.

If $b=2k+1$ is odd (in particular, $b\geq 3$), we write
$b\beta=(b\beta-e_2)+e_2$ and have
\begin{displaymath}
\mathfrak{g}_{b\beta}V_{\lambda}=
[\mathfrak{g}_{b\beta-e_2},\mathfrak{g}_{e_2}]V_{\lambda}\subset
\mathfrak{g}_{b\beta-e_2}V_{\lambda+e_2}+
\mathfrak{g}_{e_2}\mathfrak{g}_{b\beta-e_2}V_{\lambda}.
\end{displaymath}
Similarly to the argument in the previous paragraph we have
$\mathfrak{g}_{b\beta-e_2}V_{\lambda}=0$. Note that
$\mathfrak{g}_{e_1+e_2}V_{\lambda+e_2}\subset
V_{\lambda+e_1+2e_2}=0$,
$\mathfrak{g}_{e_1-e_2}V_{\lambda+e_2}\subset V_{\lambda+e_1}=0$
and that
$b\beta-e_2=(4k+2)e_1+2ke_2=(3k+1)(e_1+e_2)+(k+1)(e_2-e_2)$. Then
we have $\mathfrak{g}_{b\beta-e_2}V_{\lambda+e_2}=0$, and further
$\mathfrak{g}_{b\beta}V_{\lambda}=0$. This completes the proof.
\end{proof}

\subsection{Mixed modules}\label{s3.3}

Another way to generalize the results of this paper is to study the
supports of the so-called mixed modules. A weight
$\mathfrak{g}$-module $V$ is called {\em mixed} (see \cite{Ma3})
provided that there exists $\lambda\in\mathrm{supp}(V)$ and
$\alpha\in\mathbb{Z}^n$ such that $\dim V_{\lambda}=\infty$ and
$\dim V_{\lambda+\alpha}<\infty$. In \cite{MZ} it is shown that for
the Virasoro algebra mixed modules do not exist. However, for $n>1$
simple highest weight $W_n$-modules are mixed in the general case
(this follows for example from \cite{HWZ}). Hence it is natural to
ask whether there are other classes of simple mixed modules, for
example if there are mixed dense or mixed punctured modules.

\begin{conjecture}\label{conj}
Any mixed $\mathfrak{g}$-module is cut.
\end{conjecture}

Below we give some motivation and (weak) evidence for this
conjecture.  Denote by $\mathrm{supp}^{\infty}(V)$ the set of all
$\lambda\in\mathfrak{g}_0^*$ such that $\dim V_{\lambda}=\infty$.
For a simple weight $\mathfrak{g}$-module $V$ we also denote by
$\mathrm{supp}^{\mathrm{fin}}(V)$ the set of all
$\lambda\in\mathrm{supp}(V)+\mathbb{Z}^n$ such that $\dim
V_{\lambda}<\infty$. Note that $\mathrm{supp}^{\mathrm{fin}}(V)$ may
not be a subset of $\mathrm{supp}(V)$. Let
$\Lambda=\mathrm{supp}(V)+\mathbb{Z}^n$.

\begin{lemma}\label{lem071}
Let $V$ be a simple mixed module and $\mu\in \Lambda$ be a convex
linear combination of some elements from
$\mathrm{supp}^{\mathrm{fin}}(V)$. Then we have either $\mu=0$ or
$\mu\in \mathrm{supp}^{\mathrm{fin}}(V)$.
\end{lemma}

\begin{proof}
Let $\mu$ be a convex linear combination of some
$\mu_1,\dots,\mu_k\in \mathrm{supp}^{\mathrm{fin}}(V)$ as in the
proof of Lemma~\ref{lem34}. Assume $\mu\in \mathrm{supp}^{\infty}(V)$.
Then $\dim V_{\mu}=\infty$ while $\dim V_{\mu_i}<\infty$ for all
$i=1,\dots,k$. Hence there exists $v\in V_{\mu}$, $v\neq 0$, such
that $\mathfrak{g}_{\mu_i-\mu}v=0$ for all $i=1,\dots,k$. Repeating
the arguments from the proof of Lemma~\ref{lem34} we obtain
$\mathfrak{g}_{0}v=0$, which implies $\mu=0$. The claim follows.
\end{proof}

\begin{corollary}\label{cor072}
Any simple mixed module $V$ containing a generalized highest weight
element is a cut module.
\end{corollary}

\begin{proof}
We may assume that every element in $V$ is a generalized highest
weight element with respect to the basis $e_1,\dots,e_n$. Let
$\lambda\in \mathrm{supp}^{\mathrm{fin}}(V)$. Then for any
$\alpha\in\mathbb{N}^n$ we have $\lambda+\alpha\in
\mathrm{supp}^{\mathrm{fin}}(V)$ for otherwise $\dim
V_{\lambda+\alpha}=\infty$ and hence $V_{\lambda+\alpha}$ must
contain a nonzero vector $v$, annihilated by
$\mathfrak{g}_{-\alpha}$. Using that $v$ is a generalized highest
weight element and $\mathfrak{g}_{-\alpha}v=0$ one shows that
$\mathfrak{g}v=0$, implying that $V$ is trivial and hence not
mixed, a contradiction.

From  $\lambda+\alpha\in \mathrm{supp}^{\mathrm{fin}}(V)$ for any
$\alpha\in\mathbb{N}^n$, similarly to the proof of Lemma~\ref{lem102}
one shows that $V$ is not dense. Hence the claim follows from
Theorem~\ref{thm31}.
\end{proof}

\begin{corollary}\label{cor073}
Let $n=2$. Then any simple mixed module $V$ for which
$|\mathrm{supp}^{\mathrm{fin}}(V)|>1$ is a cut module.
\end{corollary}

\begin{proof}
Using Lemma~\ref{lem071}, similarly to the proof of
Corollary~\ref{cor071} we may assume
$\lambda\in \mathrm{supp}^{\infty}(V)$,
$\lambda+e_1\in \mathrm{supp}^{\mathrm{fin}}(V)$
and either $\lambda+e_1+e_2\in \mathrm{supp}^{\mathrm{fin}}(V)$
or $\lambda+e_1+2e_2\in \mathrm{supp}^{\mathrm{fin}}(V)$.

In the first case from $\dim V_{\lambda}=\infty$ and
$\dim V_{\lambda+e_1},\dim V_{\lambda+e_1+e_2}<\infty$ we have that there
exists $v\in V_{\lambda}$, annihilated by both $\mathfrak{g}_{e_1}$
and $\mathfrak{g}_{e_1+e_2}$. As $\{e_1,e_1+e_2\}$ is a
$\mathbb{Z}$-basis of $\mathbb{Z}^2$, the element $v$ is a
generalized highest weight element and the claim follows from
Corollary~\ref{cor072}.

In the second case ($\lambda+e_1+2e_2\in \mathrm{supp}^{\mathrm{fin}}(V)$)
one proves the existence of a generalized highest weight element
in $V_{\lambda}$ similarly to the proof of Corollary~\ref{cor072}.
The claim follows.
\end{proof}

\begin{corollary}\label{cor074}
Let $n=2$. Then for any simple mixed punctured module $V$
we have $\mathrm{supp}^{\infty}(V)=\mathrm{supp}(V)$.
\end{corollary}

\begin{proof}
This follows from Corollary~\ref{cor073} and definitions.
\end{proof}

\vspace{0.5cm}

\begin{center}
\bf Acknowledgments
\end{center}

The research was done during the visit of the first author to
Wilfrid Laurier University in April and May 2009. The hospitality
and financial support of Wilfrid Laurier University are gratefully
acknowledged. The first author was partially
supported by the Swedish Research Council. The second
author was partially supported by NSERC
and NSF of China (Grant 10871192).
We thank Svante Janson for helpful discussions.

\vspace{1cm}

\noindent
V.M.: Department of Mathematics, Uppsala University, Box 480,
SE-751 06, Uppsala, SWEDEN; e-mail: {\tt mazor\symbol{64}math.uu.se}
\vspace{0.2cm}

\noindent K.Z.: Department of Mathematics, Wilfrid Laurier
University, Waterloo, Ontario, N2L 3C5, Canada; and Institute of
Mathematics, Academy of Mathematics and System Sciences, Chinese
Academy of Sciences, Beijing 100190, PR China; e-mail: {\tt
kzhao\symbol{64}wlu.ca}


\begin{thebibliography}{99999}
\bibitem[BZ]{BZ} {\it Y.~Billig and  K.~Zhao}
Weight modules over exp-polynomial Lie algebras.
J. Pure Appl. Algebra {\bf 191} (2004), no. 1-2, 23--42.
\bibitem[DMP]{DMP} {\it I.~Dimitrov, O.~Mathieu, I.~Penkov},
On the structure of weight modules.  Trans. Amer. Math. Soc.
{\bf 352}  (2000),  no. 6, 2857--2869.
\bibitem[ER]{Ra} {\it S.~Eswara Rao}, Irreducible representations of
the Lie-algebra of the diffeomorphisms of a $d$-dimensional torus.
J. Algebra  {\bf 182}  (1996),  no. 2, 401--421.
\bibitem[Fe]{Fe} {\it S.~Fernando}, Lie algebra modules with
finite-dimensional weight spaces. I.  Trans. Amer. Math. Soc.
{\bf 322}  (1990),  no. 2, 757--781.
\bibitem[Fu]{Fu} {\it V.~Futorny}, Weight representations of
semi-simple finite-dimensional Lie algebras, Ph.D. Thesis, Kyiv
University, Kyiv, Ukraine, 1986.
\bibitem[Ga]{Ga} {\it P.~Gabriel}, Lectures at the Seminaire Godement,
Paris, unpublished notes, 1959.
\bibitem[HWZ]{HWZ} {\it J.~Hu, X.~Wang, K.~Zhao}, Verma modules over
generalized Virasoro algebras ${\rm Vir}[G]$.  J. Pure Appl. Algebra
{\bf 177}  (2003),  no. 1, 61--69.
\bibitem[La]{La} {\it N.~Lauritzen}, Lectures on convex sets, notes
available online from:  http://home.imf.au.dk/niels/lecconset.pdf
\bibitem[LZ]{LZ} {\it R.~Lu and  K.~Zhao}, Classification of irreducible
weight modules over higher rank Virasoro algebras, Adv. Math.
{\bf 201} (2006), no. 2, 630--656.
\bibitem[M1]{M}
{\em O.~Mathieu}, Classification of Harish-Chandra modules over the
Virasoro Lie algebra. Invent. Math. {\bf 107} (1992), no. 2, 225--234.
\bibitem[M2]{M2}
{\em O.~Mathieu}, Classification of irreducible weight modules.
Ann. Inst. Fourier (Grenoble)  {\bf 50}  (2000),  no. 2, 537--592.
\bibitem[Ma1]{Ma1}
{\em V.~Mazorchuk}, Futorny theorem for generalized Witt algebras of
rank $2$. Comm. Algebra {\bf 25} (1997), no. 2, 533--541.
\bibitem[Ma2]{Ma2}
{\em V.~Mazorchuk}, On the support of irreducible modules over the Witt-Kaplansky algebras of rank $(2,2)$.  Mathematika  {\bf 45}
(1998),  no. 2, 381--389.
\bibitem[Ma3]{Ma25} {\em V.~Mazorchuk}, Classification of simple
Harish-Chandra modules over $Q$-Virasoro algebra.  Math. Nachr.
{\bf 209}  (2000), 171--177.
\bibitem[Ma4]{Ma3}
{\em V.~Mazorchuk}, On simple mixed modules over the Virasoro algebra.
Mat. Stud.  {\bf 22}  (2004),  no. 2, 121--128.
\bibitem[Ma5]{Ma5}
{\em V.~Mazorchuk}, Lectures on $\mathfrak{sl}_2(\mathbb{C})$-modules,
to be published by Imperial College Press.
\bibitem[MZ]{MZ}
{\em V.~Mazorchuk and K.~Zhao}, Classification of simple weight
Virasoro modules with a finite-dimensional weight space.
J. Algebra {\bf 307} (2007), no. 1, 209--214.
\bibitem[PZ]{PZ}
{\em J.~Patera and H.~Zassenhaus}, The higher rank Virasoro
algebras, Comm. Math. Phys. {\bf 136} (1991), 1--14.
\bibitem[PS]{PS}
{\em I.~Penkov, V.~Serganova}, Weight representations of the polynomial
Cartan type Lie algebras $W\sb n$ and $\overline S\sb n$.  Math. Res. Lett.
{\bf  6}  (1999),  no. 3-4, 397--416.
\bibitem[Sh]{Sh} {\it G.~Shen}, Graded modules of graded Lie algebras of
Cartan type. I. Mixed products of modules.  Sci. Sinica Ser. A
{\bf 29}  (1986),  no. 6, 570--581.
\bibitem[Su]{Su} {\it Y.~Su}, Simple modules over the high rank
Virasoro algebras.  Comm. Algebra  {\bf 29}  (2001),  no. 5, 2067--2080.
\end{thebibliography}
\end{document}